\documentclass[11pt,reqno]{article}

\usepackage[LCY,T1]{fontenc}

\usepackage[main=english,russian]{babel}
\usepackage{amsfonts}
\usepackage{amscd}
\usepackage{color}
\usepackage{todonotes}
\usepackage{amsmath}
\usepackage{float}
\usepackage{amsthm}
\usepackage{amssymb}
\usepackage{mathrsfs}
\usepackage{amstext}
\usepackage{graphicx}
\usepackage{tikz}
\usetikzlibrary{matrix}

\usepackage{verbatim}

\usepackage{mathtools}
\mathtoolsset{showonlyrefs}

\numberwithin{equation}{section}

\theoremstyle{definition}

\theoremstyle{plain}
\newtheorem{satz}{Theorem}[section]
\newtheorem{defi}[satz]{Definition}
\newtheorem{cor}[satz]{Corollary}
\newtheorem{lem}[satz]{Lemma}
\newtheorem{prop}[satz]{Proposition}

\newtheorem*{rem*}{Remark}

\newtheorem{theorem}[satz]{Theorem}

\newcommand{\re}{\ensuremath{\mathbb{R}}}
\newcommand{\R}{\ensuremath{{\re}^d}}

\newcommand{\N}{\ensuremath{\mathbb{N}}}
\newcommand{\n}{\ensuremath{{\N}_0}}
\newcommand{\nd}{\ensuremath{\n^d}}

\newcommand{\zz}{\ensuremath{\mathbb{Z}}}
\newcommand{\Z}{{\ensuremath{\zz}^d}}

\newcommand{\C}{\ensuremath{\mathbb{C}}}

\newcommand{\tor}{\ensuremath{\mathbb{T}}}

\newcommand{\E}{\ensuremath{\mathbb{E}}}
\newcommand{\p}{\ensuremath{\mathbb{P}}}

\newcommand{\ca}{\ensuremath{\mathcal A}}

\newcommand{\bk}{\ensuremath{\mathbf k}}

\newcommand{\bj}{\ensuremath{\mathbf j}}

\newcommand{\bN}{\ensuremath{\mathbf N}}
\newcommand{\bx}{\ensuremath{\mathbf x}}
\newcommand{\by}{\ensuremath{\mathbf y}}

\newcommand{\bw}{\ensuremath{\mathbf w}}

\newcommand{\Span}{{\rm span \, }}

\newcommand{\bproof}{\begin{proof}}
\newcommand{\eproof}{\end{proof}}

\newlength{\fixboxwidth}
\newcommand{\be}{\begin{equation}}
\newcommand{\ee}{\end{equation}}
\newcommand{\beq}{\begin{eqnarray}}
\newcommand{\beqq}{\begin{eqnarray*}}
\newcommand{\eeq}{\end{eqnarray}}
\newcommand{\eeqq}{\end{eqnarray*}}

\makeatletter
\def\blfootnote{\xdef\@thefnmark{}\@footnotetext}
\makeatother

\usepackage{hyperref}
\begin{document}







\title{High-dimensional sparse trigonometric approximation in the uniform norm and consequences for sampling recovery}

\author{Moritz Moeller $\!\,\!{}^{a,}$\footnote{Corresponding author, Email:
moritz.moeller@math.tu-chemnitz.de} \hspace{1pt},
Serhii Stasyuk $\!\,\!{}^{a,b,c}$,
Tino Ullrich $\!\,\!{}^{a}$\\\\
$\!\!{}^{a}\!$ Chemnitz University of Technology, Faculty of Mathematics\\\\
$\!\!{}^{b}\!$ Leipzig University, Institute of Mathematics\\\\
$\!\!{}^{c}\!$ Institute of Mathematics of NAS of Ukraine}

\date{July 22, 2026}\blfootnote{\textit{Keywords and phrases:} multivariate approximation; best \(m\)-term approximation; uniform norm; rate of convergence; sampling recovery.}\blfootnote{\textit{2020 Mathematics subject classification:}  42A10, 94A20, 41A46, 46E15, 42B35, 41A25, 41A17, 41A63}
\maketitle

\begin{abstract} Recent findings by Jahn, T. Ullrich, Voigtlaender \cite{JUV23} relate non-linear sampling numbers for the square norm to quantities involving trigonometric best $m$-term approximation errors in the uniform norm. Here we establish new results for sparse trigonometric 
approximation in the high-dimensional setting, where the influence of the dimension $d$ has to be controlled. In particular, we focus on best $m$-term trigonometric approximation widths for (unweighted) Wiener classes in $L_q$ and give precise constants. Our main results are approximation guarantees where the number of terms $m$ scales at most quadratically in the inverse accuracy $1/\varepsilon$. Providing a refined version of the classical Nikol'skii inequality, we are able to extrapolate the $L_q$-result to $L_\infty$ while limiting the influence of the dimension mainly to a $\log$-term in the support size of the (rectangular) spectrum. This has consequences for the tractable sampling recovery via $\ell_1$-minimization of functions belonging to certain Besov classes with mixed smoothness and complements polynomial tractability results recently given by Krieg \cite{Kr23}.  
\end{abstract}

\section{Introduction}

Reconstructing an unknown multivariate function from a limited amount of information, in particular a finite number of function samples is a basic problem in approximation theory and numerical analysis. When the number of variables $d$ is large, this task is threatened by the curse of dimensionality. Without additional structure, the amount of information required to guarantee a prescribed accuracy usually grows exponentially in $d$. 

In this paper we will examine specific function classes built upon an $\ell_1$-sparsity condition, namely Wiener and related Besov spaces, where the curse can be avoided to some extent. We observe a polynomial dependence on the dimension $d$ of the function recovery error (standard information) in $L_2$ which relates to polynomial tractability. We study this recovery problem through its recently discovered relation to sparse trigonometric approximation in $L_\infty$, and we show that for the mentioned classes  of functions the influence of the dimension can be kept under control.

To make things precise, let $Y$ denote a complex quasi-Banach space and $\mathcal{D} \subset Y$ a dictionary and $f \in Y$. We consider the best $m$-term approximation error of $f$ in $Y$ with respect to $\mathcal{D}$ 
\begin{equation}\label{f001}
    \sigma_m(f;\mathcal{D})_Y := \inf\limits_{\substack{e_i \in \mathcal{D}, c_i \in \mathbb{C}\\i=1,...,m}} \Big\|f-\sum\limits_{i=1}^m c_i e_i\Big\|_Y\,.
\end{equation}

In this paper we are particularly interested in $L_q$-approximation (including $q=\infty$) of multivariate functions on the $d$-torus $\mathbb{T}^d=[0,1)^d$ with respect to the dictionary $\mathcal{D}:=\mathcal{T}^d$ consisting of complex exponentials. Trigonometric best $m$-term approximation of multivariate functions $f$ with respect to a given dictionary $\mathcal{D}$ is a  classical subject in approximation theory, see e.g. \cite[Chapt.\ 7] {DuTeUl18} for references and historical remarks. 

It turned out only recently that there is an intimate relation between corresponding best $m$-term widths in the uniform norm and non-linear sampling numbers in the square norm \cite{JUV23}. Non-linear sampling numbers $\varrho_m$ are asymptotic characteristics which characterize the worst-case error for the recovery of functions $f \in \mathbf{F}$ using only $m$ function values (standard information), see \cite[Sect.\ 4.1.1 and~4.1.4]{NovakWo2008}. Sampling recovery of periodic functions with bounded mixed derivative/difference has been studied intensively in the Russian literature, see, e.g., Temlyakov's monograph \cite[Chapt.\ 4]{Tem93_book}.  

The  general sampling recovery problem has been formulated by Wasilkowski and Wo\'{z}niakowski in \cite{WaWo01} who compared tractability and strong tractability between standard and general information. Kuo, Wasilkowski and Wo\'{z}niakowski \cite{KWW09} extended this line of research and covered $L_q$-approximation for the full range $1\leq q\leq \infty$. For the uniform norm, however, they could only locate the optimal order of convergence of standard information within an interval. This gap was closed by Pozharska and T.\ Ullrich in \cite{PoUl22}.

One main result in \cite[Thm.\ 3.1]{JUV23}, linking trigonometric approximation and the non-linear sampling numbers is
\begin{equation}\label{eq:intro1}
  \begin{split}
     &\varrho_{\lceil Cd (\log{d}) m (\log{m})^2\log(N)\rceil } (\mathbf{F})_{L_2(\tor^d)}\\ 
     &~~~~\leq \widetilde{C}
           \big(\sigma_m(\mathbf{F};\mathcal{T}^d)_{L_\infty(\tor^d)}
                   + E_{\mathcal{T}([-N,N]^d)}(\mathbf{F})_{L_\infty(\tor^d)}
                 \big)\,
	\end{split}
\end{equation}
with absolute constants $C, \widetilde{C}>0$. The recovery algorithm is based on $\ell_1$-minimization used in compressed sensing. That is why we have to use the (general non-linear) sampling numbers
\begin{equation*}
  \varrho_m (\mathbf{F})_Y
  := \inf_{t_1,\ldots,t_m \in \Omega}\,
       \inf_{R : \mathbb{C}^m \to Y}\,
         \sup_{\|f\|_{\mathbf{F}} \leq 1}\,
           \|f - R(f(t_1),\ldots,f(t_m))\|_Y\,
\end{equation*}
with possibly non-linear recovery maps $R:\mathbb{C}^m \to Y$. The right-hand side in \eqref{eq:intro1} consists of two summands, a non-linear and a linear best approximation error. The non-linear width dominates the bound. The result above has been recently refined  by Dai, Temlyakov \cite{DaTe24}  using  greedy algorithms. Furthermore, Krieg \cite{Kr23} established a slightly better bound which is based on earlier work by Rauhut, Ward \cite{RaWa16}. Finally, Moeller, Pozharska and T. Ullrich \cite[Theorem 3]{MPU25_1} showed an instance optimal version of this result that holds for \(2 \leq q  \leq \infty\) in the left-hand side. This was achieved by using \emph{Square Root Lasso} and \emph{Orthogonal Matching Pursuit} instead of $\ell_1$-minimization (basis pursuit denoising), see also \cite[Section 5]{MNPSU26} for further discussion and numerical implementation.

In fact, the best $m$-term approximation error can be considered in the Wiener algebra norm with an additional factor $m^{-1/2}$. The question arose whether already best $m$-term approximation in $L_\infty$ performs equally well in high dimensions. We give an affirmative answer in this paper. The corresponding bounds can then be rephrased for the sampling recovery problem taking the index for $\varrho_m$ in \eqref{eq:intro1} into account.

Our starting points are (unweighted) Wiener spaces $\mathcal{A}(\tor^d)$ (see Definition \ref{defi_uww}) defined via the $\ell_1$-summability condition of the Fourier coefficients and its extensions to $\ca_\theta$ with $\theta<1$. We give results for the $L_q$-norm with precisely controlled constants scaling polynomially in $d$ and $q$. The classes gained significant interest in recent years  partly due to their relation to the celebrated Barron class relevant for neural network approximation, see  Barron \cite{Ba93, Ba94} and Voigtlaender \cite{Vo22}. In fact, the enforced sparsity through the $\ell_\theta$-condition ensures a faster decay of certain asymptotic characteristics. Various \(s\)-numbers, sparse grid sampling, numerical integration and first results on best \(m\)-term (approximation) widths have been studied for their weighted counterparts by V. K. Nguyen, V. D. Nguyen and Sickel \cite{NNS22}, V. K. Nguyen and V. D. Nguyen \cite{NN22}, Kolomoitsev, Lomako, and Tikhonov \cite{KoLoTi23}, Chen and Jiang \cite{ChenJiang24}, Moeller \cite{Mo23} and Temlyakov \cite{Te23}. In our paper \cite{MSU25} we extend our results to the weighted setting by employing the unweighted results shown in this paper and give near optimal asymptotic bounds.

In order to obtain results for the uniform norm we rely on a classical strategy from probability theory which has been adapted to problems in approximation theory by Trigub and Belinsky~\cite[Theorem 11.2.4]{TrBe04}. First, we bound the best $m$-term error in the $L_q$-norm for large but finite $q$, tracking precisely how the constant grows with $q$ and $d$ (Theorem~\ref{tracAq}). Second, a sharpened Nikol'skii inequality tailored to this purpose (Theorem~\ref{Niko}) lets us pass from $L_q$ to $L_\infty$ by choosing $q$ large, at the cost of only a logarithmic factor in the size of the spectrum. Third, \eqref{eq:intro1} transfers the resulting $L_\infty$-bounds to non-linear sampling numbers. Along this route the influence
of $d$ is confined to small polynomial and logarithmic factors, showing that this problem is indeed tractable.

We begin by introducing the newly tailored version of the Nikol'skii inequality proved in this paper, see Theorem \ref{Niko} below. It essentially says for $2\leq q<\infty$,
\be\label{intr1}
\Vert f \Vert_{L_\infty} \leq e\, \Vert f \Vert_{L_q} (2d+2)^{\frac{d}{q}} \prod_{j=1}^d N_j^{\frac{1}{q}}\,,
\ee
where $f\in \mathcal{T}\big([-\bN,\bN]\big)$ with $\bN = (N_1,...,N_d) \in \mathbb{N}^d$\,. The embedding constant is neutralized if $q$ is chosen large depending on the spectrum $[-\bN,\bN]$ and $d$.

Asymptotic estimates for $\sigma_m(f;\mathcal{T}^d)_{L_q}$ for trigonometric polynomials $f \in \mathcal{A}(\tor^d)$ already exist, see for instance DeVore, Temlyakov \cite{DeTe95} and Temlyakov \cite{Tem15,Teml_CA_2017}\,. However, to the best of the authors' knowledge the behavior of the $d$-dimensional constants is still unknown, see also the remark after Theorem  \ref{tracAq}. Note that in the setting of Wiener classes it is a natural question to ask for the dependence of the constant on $d$. Clearly, these constants would significantly depend on the choice of the equivalent norm in the space of interest. However, for Wiener spaces and the target space $L_\infty$ there is a natural choice of the norm. We prove the following tractable result in Theorem \ref{tracAq} below. For \(2 \leq q < \infty\), \(m \geq 4q\) and \(0 <\theta \leq 1\) it holds
\begin{equation}\label{intr2}
\frac{3^{-1/2-1/\theta}}{2} m^{-(1/\theta-1/2)} \leq  \sigma_{m}(\mathcal{A}_\theta;\mathcal{T}^d)_{L_q} \leq 
  C\, \sqrt{q}\,4^{1/\theta-1/2} m^{-(1/\theta-1/2)}
,
\end{equation}

 \noindent 
 where $C <  16$ (without any $d$-depending constant). Our technique is probabilistic and uses ideas from \cite{Mak84}. However, it is elementary and self-contained via applying the probabilistic Bernstein inequality which makes it possible to keep track of all the involved constants. 
 This is essential for the formulation of tractable results.

 In addition, Theorem \ref{extend_trac} also provides sharp bounds for \(1 \leq \theta < q'\), here it holds
  \begin{equation}\label{intr2_theta<q'}
 \frac{1}{ 6 (\sqrt{3}\, q + 1) } m^{ -\frac{q}{2}\big(\frac{1}{q}+\frac{1}{\theta}-1\big)} \leq \sigma_{m}(\ca_\theta;\mathcal{T}^d)_{L_q} \leq 
       C_1 \cdot \Big(\frac{m}{4q}\Big)^{-\frac{q}{2}\left(\frac{1}{q}+\frac{1}{\theta}-1\right)}  
  \end{equation} 
  
 \noindent  
 with some absolute constant \(C_1  < 47 \).
The results corresponding to \eqref{intr2} and \eqref{intr2_theta<q'} for $\mathbf{B}^{1/\theta-1/2}_{2,\theta}$ are stated in Theorems \ref{thm:5.3} and \ref{extend_trac_B}.

 \medskip
 
 Combining \eqref{intr1} and \eqref{intr2}, Theorem \ref{A_infty} gives for trigonometric polynomials $t \in \mathcal{T}(Q)$ with frequencies in a rectangle $Q$ 
\be \label{intr3}
\sigma_{4m}(t;\mathcal{T}^d)_{L_\infty} \leq    C_2 \log(d)^{1/2}\, m^{-(1/\theta-1/2)} \big(\log {\# Q} \big)^{1/2}\Vert t\Vert_{\ca_\theta},
\ee 
with an absolute constant 
 \(C_2 < 295\) and \( {\# Q} \) denotes the integer grid cardinality of the rectangle $Q$. Note that this result only holds for functions with bounded support on the frequency side. Nonetheless, since the size of the support set enters only in the logarithm, we may choose it rather large (depending on $m$) such that it will affect the final estimate only by $\sqrt{\log m}$-term. This result is again tractable in the sense that the dimension enters at worst polynomially (through the term \(\log(\# Q)^{1/2}\)).

In the remaining part of the paper we give some consequences for Besov spaces with dominating mixed smoothness (spaces with bounded mixed difference), see \cite[Chapt.\ 3]{DuTeUl18}. The result is worth mentioning since these spaces represent classical smoothness spaces which may be characterized by differences and derivatives. A special case of Theorem \ref{besov} below states the following new result in $L_\infty$ for limiting mixed smoothness $r=1/2$ when $2<p<\infty$:

 \be\label{intr4}
 \sigma_{4m}(\mathbf{B}^{1/2}_{p,1};\mathcal{T}^d)_{L_\infty} \leq C_3\, d\log(d)^{1/2}\, \sqrt{\frac{4}{p-2}} \, m^{-1/2} \log(dm)^{1/2}\,,
 \ee 
 for some absolute constant \(C_3 <304\). The result shows that the influence of the dimension is at worst of order $\mathcal{O}(d)$ up to logarithmic factors. In addition, it shows a rate which decays faster than the corresponding Kolmogorov width in $L_2$ when $d>2$, see \cite[Thm.\ 4.5.3]{DuTeUl18} and the references therein. This effect transfers to the sampling numbers via \eqref{eq:intro1}. Higher rates of convergence can be achieved by passing to quasi-Banach spaces with $\theta < 1$, where we obtain the main rate $m^{-(1/\theta-1/2)}$ for $\sigma_{m}(\mathbf{B}^{1/\theta-1/2}_{p,\theta};\mathcal{T}^d)_{L_\infty}$. Note that the fine index $\theta\leq 1$ enforces the necessary sparsity which leads to the good estimates in high dimensions. When considering Besov-Nikol'skii spaces with mixed smoothness (where $\theta = \infty$) or corresponding Sobolev spaces we observe an additional $\log(m)^{d-1}$-term, see Temlyakov, T. Ullrich \cite{TeUl22}\,.      

{\bf Notation.} As usual $\N$ denotes the natural numbers, $\N_0=\N\cup\{0\}$,
$\zz$ denotes the integers,
$\re$ the real numbers, and $\C$ the complex numbers. The letter $d$ is always
reserved for the underlying dimension in $\re^d, \zz^d$ etc. For $a\in \re$ we denote $a_+ \coloneqq \max\{a,0\}$ and for \(\bx \in \re^d\) we define this pointwise \( \bx_+ \coloneqq ((x_1)_+,\ldots,  (x_d)_+)\) and by $ \bx = (x_1,\ldots,x_d)>0$ we mean that each coordinate is positive.
For $0<p\leq \infty$ and $\bx\in \R$ we denote $|\bx|_p = (\sum_{i=1}^d |x_i|^p)^{1/p}$ with the
usual modification in the case $p=\infty$. By \(\bx\cdot \by \) or \(\bx \by \) we indicate the inner product. If $X$ and $Y$ are two (quasi-)normed spaces, the (quasi-)norm of an element $x$ in $X$ will be denoted by \(\Vert x\Vert_X\). We use the usual notations \(L_p\) for the Lebesgue (quasi-)Banach spaces and \(\ell_p\) for sequence spaces.
The symbol $X \hookrightarrow Y$ indicates that the
identity operator is continuous and \(X'\) denotes the dual space of \(X\). For two sequences $a_n$ and $b_n$ we will write
$a_n \lesssim b_n$ if there exists a constant $c>0$ and \(n_0 \in \N \) such that $a_n \leq c\,b_n$
for all $n \geq n_0$. We will write $a_n \asymp b_n$ if $a_n \lesssim b_n$ and $b_n
\lesssim a_n$.  
The torus \(\tor\) represents the segment \( [0,1)\) where we identify the endpoints. The $d$-torus $\tor^d$ denotes its $d$-fold tensor product. Finally, with \(\log(x)\) we denote the natural logarithm of $x>0$. Where \(\log(d)\) appears this term should be replaced by \(1\) for \(d=1\).


\section{Trigonometric best $m$-term widths and Wiener spaces}
For an integrable $d$-variate periodic function $f:\tor^d \to \mathbb{C}$ we define the Fourier coefficient with respect to the multi-indexed frequency $\mathbf{k}\in \zz^d$ as usual by
$$
    \hat{f}(\mathbf{k}) := \int_{\tor^d} f(\bx)\exp(-2\pi \mathrm{i} \bk\cdot \bx) \, \mathrm{d} \bx\,.
$$
An important role in this paper will be played by trigonometric polynomials. For a subset $Q \subset \zz^d$ we define the linear set
$$
    \mathcal{T}(Q):=\Big\{\sum\limits_{\bk \in Q} c_{\bk} \exp(2\pi \mathrm{i} \bk\cdot) :~c_{\bk} \in \mathbb{C}, \bk \in Q\Big\}\,.
$$
Let us single out the special case 
$$
   \mathcal{T}\big([-\bN,\bN]\big) := \mathcal{T}(\bigtimes_{j =1}^d [-N_j,N_j])\,,
$$
where $Q = \bigtimes_{j =1}^d [-N_j,N_j]$ represents an (anisotropic) cuboid\,. The spaces  $\mathcal{T}(Q)$ are linear subspaces of $\Span \mathcal{T}^d$, where $\mathcal{T}^d$ denotes the dictionary consisting of all complex exponentials, i.e.,
$$
    \mathcal{T}^d := \Big\{\exp(2\pi \mathrm{i} \mathbf{k}\cdot) :~\mathbf{k} \in \mathbb{Z}^d\Big\}\,.
$$
Let us now comment on non-linear approximation with respect to $\mathcal{T}^d$ and replace the above linear spaces by the non-linear space 
\begin{equation}
   \Sigma_m:=\Big\{\sum\limits_{j=1}^m c_j e_j(\cdot) :~e_1(\cdot),...,e_m(\cdot) \in \mathcal{T}^d; \, c_1,...,c_m \in \mathbb{C}\Big\} \subset \Span \mathcal{T}^d\,.
\end{equation}   
We denote by $\sigma_m(f;\mathcal{T}^d)_{L_q}$ the best approximation error $\|f-s\|_{L_q}$ using a trigonometric $m$-term sum 
$$
        s = \sum\limits_{\mathbf{k} \in \Lambda} c_{\mathbf{k}}\exp(2\pi \mathrm{i} \mathbf{k} \cdot \mathbf{x})\in \Sigma_m
$$
with an appropriate $\Lambda \subset \zz^d$\,.
Having a function class $\mathbf{F}$ (unit ball of some quasi-Banach space of multivariate functions) we define 
\begin{equation}\label{eq3}
    \sigma_m(\mathbf{F};\mathcal{D})_Y := \sup\limits_{f \in \mathbf{F}}\sigma_m(f;\mathcal{D})_Y
\end{equation}
and call it best $m$-term width of $\mathbf{F}$ in $Y$ with respect to the dictionary $\mathcal{D}$. Characteristics of this type were first introduced by Stechkin \cite{Ste55} in 1955 and have been studied by Pietsch \cite{Pie81} in the 1980s. Non-linear approximation gained a lot of interest when wavelets entered image and signal processing, see DeVore \cite{DeV98} for references and historical remarks. 

\begin{defi}\label{defi_uww} Let $0<\theta \leq\infty$. The space $\mathcal{A}_\theta$ is defined as
$$
    \mathcal{A}_{\theta} \coloneqq\mathcal{A}_{\theta}(\tor^d):=\Big\{f \in L_1(\tor^d)~:~\|f\|_{\mathcal{A}_\theta}:=\Big(
    \sum\limits_{\mathbf{k}\in \zz^d} |\hat{f}(\mathbf{k})|^{\theta}\Big)^{1/\theta} < \infty\Big\}\,,
$$
with the usual modifications in the case where \(\theta = \infty\). The quantity $\|f\|_{\mathcal{A}_\theta}$ is a norm on $\mathcal{A}_{\theta}$ if $\theta\geq 1$ and a quasi-norm if $\theta<1$. In case $\theta=1$ the space $\mathcal{A}:=\mathcal{A}_{1}$ is called Wiener algebra. 
\end{defi}

The following statement is a direct consequence of a result known as {\em Stechkin's lemma} although Stechkin never stated it in this form, see \cite[Section 7.4]{DuTeUl18} for historical comments and further references.
Let the Fourier coefficients $(\hat{f}(\bk_j))_j$ of a function $f \in L_1(\tor^d) $ be arranged such that $$
|\hat{f}(\mathbf{k}_1)| \geq |\hat{f}(\mathbf{k}_2)| \geq |\hat{f}(\mathbf{k}_3)| \geq \dots$$

\begin{lem} \label{stechkin} Let  \(0 < \theta  \leq \gamma  \leq \infty \). Then it holds 
\begin{equation}
     \sigma_m(f;\mathcal{T}^d)_{\mathcal{A}_\gamma} := \Big(\sum\limits_{j=m+1}^{\infty} |\hat{f}(\mathbf{k}_j)|^{\gamma}\Big)^{1/\gamma} \leq (m+1)^{-(1/\theta - 1/\gamma)}\|f\|_{\mathcal{A}_{\theta}}   
\end{equation}
for all $f\in \mathcal{A}_{\theta}$\,. As usual, if $\gamma = \infty$ the sum in the middle is replaced by $|\hat{f}(\mathbf{k}_{m+1})|$.

\end{lem}
For our purposes this classical result is sufficient, for more results concerning best \(m\)-term widths in sequence spaces, we refer to the recent work by V.K.~Nguyen and V.D.~Nguyen \cite{NN22}. Here we are interested in the situation where the target space is $L_q$, $2\leq q \leq \infty$.


\section{Nikol'skii's inequality revisited}
Nikol'skii's inequality (of different metrics) has been first mentioned in \cite[Eq.\ 2.4]{Nik51}. For the multivariate trigonometric version we refer to \cite[Thm.\ 2.4.5]{DuTeUl18} and the references therein. It essentially states
\be \label{Niko_original}
\Vert f \Vert_{L_\infty} \leq C(d)\,\Vert f \Vert_{L_q} \prod_{j=1}^d N_j^{\frac{1}{q}}
\ee
for $f \in T([-\bN, \bN])$. With the classical proof one may observe that the constant $C(d)$ is usually exponentially growing in $d$. We will sharpen this inequality what concerns the $d$-dependence. In the below version of the Nikol'skii inequality the constant (including the product of the \(N_j\)) is less than $e^3$ if $q$ is chosen as $2 \log(d)\log(\prod 2 N_j +1)$.

As in Jahn, Ullrich, Voigtlaender \cite{JUV23} we rely on mapping properties of specifically tailored de La Vall{\'e}e Poussin operators together with an interpolation argument to prove the result below. Ditzian and Tikhonov \cite{DT05} used a similar strategy to prove refined versions of Nikol'skii's inequality. In fact, their Theorem 6.3 together with the de La Vall{\'e}e Poussin operators used in our proof leads to similar results in this specific situation.

\begin{theorem}[Nikol'skii's inequality] \label{Niko}
Let \(2 \leq q < \infty\) and \(\bN \in \N^d \) then for
\[f\in \mathcal{T}\big([-\bN,\bN]\big) \coloneqq \mathcal{T}\big(\bigtimes_{j =1}^d [-N_j,N_j]\big),\]

\noindent it holds
\be
\Vert f \Vert_{L_\infty} \leq e^{1-\frac{2}{q}} \, \Vert f \Vert_{L_q} (2d+2)^{\frac{d}{q}} \prod_{j=1}^d N_j^{\frac{1}{q}}\,.
\ee

\noindent In particular, for \(f \in \mathcal{T}([-N,N]^d)\) it holds


\[\Vert f \Vert_{L_\infty} \leq e^{1-\frac{2}{q}}  \Vert f \Vert_{L_q}\left((2d+2)N\right)^{\frac{d}{q}}.\]

\end{theorem}

\begin{proof}
We use the mapping properties of a modified de La Vall{\'e}e Poussin operator (see \cite[Sect.\ 3.1]{JUV23})  
\be
V_\bN(f)(\bx) =  \sum_{\bk \in \Z} \hat{f}(\bk) \, v_\bk \exp(2 \pi \mathrm{i}  \bk\cdot \bx)
\ee
with weights \( v_\bk = \prod_{j = 1}^d v_{k_j}\), where
\be
v_{k_j} = 
\begin{cases}
1, & \vert k_j\vert \leq N_j , \\
\frac{(2d+1)N_j - \vert k_j\vert}{2dN_j} \,, & N_j <\vert k_j\vert \leq (2d+1)N_j , \\
0 , &\vert k_j\vert > (2d+1)N_j\,.
\end{cases}
\ee
  \noindent
Let us first estimate $\Vert V_\bN \Vert_{L_2 \to L_\infty}$. Afterwards, we will take care of showing \(\Vert V_\bN \Vert_{L_\infty \to L_\infty} \leq e\) and apply an interpolation argument. With a straightforward calculation we find 
\begin{align}
\Vert V_\bN \Vert_{L_2 \to L_\infty} & = \sup_{\Vert f \Vert_{L_2} \leq 1} \Vert V_\bN(f) \Vert_{L_\infty} \\
& = \sup_{\Vert f \Vert_{L_2} \leq 1} \Big\Vert \sum_{\bk \in \Z} \hat{f}(\bk) v_\bk \exp(2 \pi \mathrm{i} \bk\cdot \bx) \Big\Vert_{L_\infty}
 =\Big(\sum_{\bk \in [-(2d+1)\bN,(2d+1)\bN]} |v_{\bk}|^2\Big)^{1/2} \\
& \leq \Big(\sum_{\bk \in [-(2d+1)\bN,(2d+1)\bN]}  \prod_{j = 1}^d v_{k_j} \Big)^{1/2}  \leq \Big(\prod_{j = 1}^d (2d+2) N_j \Big)^{1/2} \\
& = (2d+2)^{\frac{d}{2}} \prod_{j = 1}^d N_j^{\frac{1}{2}}\,.
\end{align}


 \noindent 
From \cite[Sect.\ 3.1]{JUV23} we obtain additionally


  \noindent
 \begin{equation}
  \begin{split}
   \Vert V_\bN \Vert_{L_\infty \to L_\infty} 
&\leq \Big\|\sum_{\bk \in \Z} v_\bk \exp(2 \pi \mathrm{i} \bk \cdot \bx)
\Big\|_{L_1(\mathbb{T}^d)} = \prod\limits_{j=1}^d
\Big\|\sum_{j \in \zz} v_{k_j} \exp(2 \pi \mathrm{i} k_j x) 
\Big\|_{L_1(\mathbb{T})}\\
&\leq \prod\limits_{j=1}^d \Big(\frac{2(d+1)N_j}{2dN_j}\Big) \leq \Big(1+\frac{1}{d}\Big)^d \leq e. 
 \end{split}
 \end{equation}

Now we use a standard interpolation of operators argument to get results for \(2 < q < \infty\).
The interpolation parameter \(\theta\) now has to be chosen in such a way that \(\frac{1}{q}= \frac{1-\theta}{2}+ \frac{\theta}{\infty}\) which yields \(\theta = 1 - \frac{2}{q}\). This gives
\begin{align}
   \Vert V_\bN \Vert_{L_q \to L_\infty} & \leq  \Vert V_\bN \Vert_{L_2 \to L_\infty}^{1-\theta} \Vert V_\bN \Vert_{L_\infty \to L_\infty}^\theta \\
    & \leq \Vert V_\bN \Vert_{L_2 \to L_\infty}^{\frac{2}{q}} \Vert V_\bN \Vert_{L_\infty \to L_\infty}^{1-\frac{2}{q}} \\
    & \leq  e^{1-\frac{2}{q}} (2d+2)^{\frac{d}{q}} \prod_{j = 1}^d N_j^{\frac{1}{q}},
\end{align}
 \noindent
which concludes the proof. 
\end{proof}

\section{Sparse trigonometric approximation in $L_q$, $2\leq q \leq \infty$}\label{sec:4}

\begin{lem}[Bernstein inequality, {\cite[Cor.\ 7.31]{FoRa13}} ] \label{Bernstein}
    Let \(X_1, \ldots , X_J\) be independent mean-zero random variables with uniform bound \( \big\vert X_j \big\vert < B\) and controlled second moments \(\sigma^2 = \sum_j \E\big\vert X_j\big\vert^2\). Then it holds for all \(s \geq 0\)
\begin{equation}
    \p \Big( \Big\vert \sum_{j=1}^J X_j\Big\vert > s \Big) \leq 2\exp\Big(- \frac{s^2/2}{\sigma^2 + Bs/3} \Big).
\end{equation}
\end{lem}

Lemma \ref{Bernstein} also holds for complex-valued random variables with slight modifications in the constants. 
\begin{cor} \label{CBernstein}
Under the same conditions as in Lemma \ref{Bernstein} it holds for independent mean-zero complex valued random variables \(Z_1,\ldots,Z_J\) the bound,
\begin{equation}
    \p \Big( \Big\vert \sum_{j=1}^J Z_j\Big\vert > s \Big) \leq 4\exp\Big(- \frac{s^2/4}{\sigma^2 + Bs/(3\sqrt{2})} \Big).
\end{equation}
\end{cor}

\begin{proof}
    We consider that if a complex number is bounded from below by \(s\) then the absolute value of either its real or imaginary part must be bounded by \(\frac{s}{\sqrt{2}}\). From there follows this bound
\begin{align}
\begin{split}
        \p \Big( \Big\vert \sum_{j=1}^J Z_j \Big\vert > s \Big) & \leq \p \Big( \Big\vert \sum_{j=1}^J X_j \Big\vert > \frac{s}{\sqrt{2}} \cup \Big\vert \sum_{j=1}^J Y_j \Big\vert > \frac{s}{\sqrt{2}} \Big) \\
        & \leq \p \Big( \Big\vert \sum_{j=1}^J X_j \Big\vert > \frac{s}{\sqrt{2}} \Big) + \p \Big( \Big\vert \sum_{j=1}^J Y_j \Big\vert > \frac{s}{\sqrt{2}} \Big) \\
        & \leq 4 \exp \Big(- \frac{s^2/4}{\sigma^2 + Bs/(3\sqrt{2})}  \Big).
\end{split}
\end{align}
Where in the last line we used that \(X_1,\ldots, X_J\) and \(Y_1,\ldots,Y_J\) are families of real-valued independent random variables that fulfill the requirements of Lemma \ref{Bernstein}.
\end{proof}

The following is a specific formulation of ~``tails to moments'' type, adapted to our formulation of the Bernstein inequality. General formulations can be found in \cite[Section 7.2]{FoRa13}.
\begin{cor}[Tails to moments] \label{Mo_to_Ta}
    Let \( X \) be a random variable with bounded tails of the form 
    \begin{equation} \label{Tail}
    \p \Big( \big\vert X \big\vert > s \Big) \leq \beta\exp\Big(- \frac{s^2}{\sigma^2 + B s} \Big).
    \end{equation}
    Then for \(p\geq 2\)  the \(p\)-th moment of \(X\) is bounded by
\begin{equation} \label{Moment}
\E\big(\vert X \vert^p\big)^{1/p} \leq  \beta^{1/p}  \sqrt{p} \sigma + 2\beta^{1/p}  p B.
\end{equation}
\end{cor}

\begin{proof}
We first consider a different representation of the \(p\)-th moment
\begin{align}
\begin{split}
\E\big(\vert X \vert^p\big) & = p\int_0^\infty \p \big(\big\vert X \big\vert > s\big) s^{p-1} \mathrm{d}s\\
& \leq p \int_0^{\sigma^2/B} \beta\exp\Big( -\frac{s^2}{2\sigma^2} \Big) s^{p-1} \mathrm{d}s + p\int_ {\sigma^2/B}^\infty \beta\exp\Big( -\frac{s}{2B} \Big) s^{p-1} \mathrm{d}s. \\
\end{split}
\end{align}
We can see that these two summands are very similar. Indeed we can write them in a unified way by introducing an additional variable \(\gamma\) such that \(\gamma = 2\) describes the first and \(\gamma = 1\), the second term and denoting by \(\kappa_2 = \sigma\) and \(\kappa_1=B\). We can now substitute \(u = \frac{s^\gamma}{\gamma}\)
\begin{align}
\begin{split}
p\int_ {0}^\infty \beta\exp\Big( -\frac{s^\gamma}{2\kappa_\gamma^\gamma } \Big) s^{p-1} \, \mathrm{d}s & = p\beta \int_ {0}^\infty \exp\Big( -\frac{u \gamma}{2\kappa_\gamma^\gamma } \Big) (\gamma u)^{\frac{p-1}{\gamma}} (\gamma u)^{\frac{1-\gamma}{\gamma}} \, \mathrm{d}u \\
& = p \beta \gamma^{\frac{p}{\gamma}-1} \int_ {0}^\infty \exp\Big( -\frac{u \gamma}{2\kappa_\gamma^\gamma } \Big) u^{\frac{p}{\gamma}-1} \, \mathrm{d}u\,.
\end{split}
\end{align}
A standard change of variable \(t = \frac{u\gamma}{{2\kappa_\gamma^\gamma }}\) yields 
\begin{align}
\begin{split}
p \beta \gamma^{\frac{p}{\gamma}-1}\int_ {0}^\infty \exp\Big( -\frac{u \gamma}{2\kappa_\gamma^\gamma } \Big) u^{\frac{p}{\gamma}-1} \, \mathrm{d}u & = p \beta \gamma^{\frac{p}{\gamma}-1} \int_ {0}^\infty \exp( -t ) \big( t 2\kappa_\gamma^\gamma \gamma^{-1}\big)^{\frac{p}{\gamma}-1} 2\kappa_\gamma^\gamma \gamma^{-1} \, \mathrm{d}t \\
& = p \beta \gamma^{\frac{p}{\gamma}-1} \big(2\kappa_\gamma^\gamma \gamma^{-1}\big)^{\frac{p}{\gamma}} \int_{0}^\infty \exp(-t ) t^{\frac{p}{\gamma}-1} \, \mathrm{d}t \\
& =  \beta \big(2\kappa_\gamma^\gamma\big)^{\frac{p}{\gamma}} \frac{p}{\gamma} \Gamma\Big(\frac{p}{\gamma}\Big) \\
& \leq \beta \big(2\kappa_\gamma^\gamma\big)^{\frac{p}{\gamma}} \Big(\frac{p}{\gamma}\Big)^{\frac{p}{\gamma}}\,, 
\end{split}
\end{align}
where we used that \( x\Gamma(x) \leq x^x\), see e.g. \cite[Lem.\ 4.9]{KSU14}.
Now taking the \(p\)-th root yields the assertion
\begin{align}
\begin{split}
    \E\big(\vert X \vert^p\big)^{1/p} & \leq \sum_{\gamma=1}^2 \beta^{1/p}  (2p/\gamma)^{1/{\gamma}} \kappa_\gamma \\
    & \leq \beta^{1/p}\sqrt{p}\, \kappa_2 + \beta^{1/p} 2p \kappa_1 \\
    & = \beta^{1/p}\sqrt{p}\, \sigma + 2\beta^{1/p} p B.
\end{split}
\end{align}

\end{proof}

First we show that every function from the (classical) Wiener Algebra can be well approximated in the $L_q$ norm by trigonometric polynomials. We use the random choice of coefficients proposed in \cite[Lem.\ 2]{Mak84}. Our proof is rather elementary and allows for tracking the involved constants. 

\begin{theorem} \label{tracAq}
Let \(2 \leq q < \infty\), \(m \geq q\) and \(0 <\theta \leq 1\) then it holds
\be\label{thm:4.4}
\sigma_{4m}(\ca_\theta;\mathcal{T}^d)_{L_q} \leq C\, \sqrt{q}\, m^{1/2-1/\theta}\,,
\ee
where $C =\frac{24 + 16 \sqrt{2}}{3} <  16$. For \(m < q\) we instead have
\(\sigma_{4m}(\ca_\theta;\mathcal{T}^d)_{L_q} \leq C\, q\, m^{-1/\theta}\).

In addition, we have
\be\label{thm:4.4_1}
\sigma_{m}(\ca_\theta;\mathcal{T}^d)_{L_q} \geq  \frac{3^{-1/2-1/\theta}}{2} m^{-(1/\theta-1/2)}\,.
\ee
\end{theorem}




\begin{proof} 
{\em Step 1.} We  first consider the approximation of a trigonometric polynomial \(t \in \Span \mathcal{T}^d \cap \mathcal{A}_{\theta}\) by a trigonometric polynomial \(s \in \Sigma_{4m}\) with at most \(4m\) non-zero frequencies. Let $t$ be given as follows 
\be \label{fdecomp}
t(\bx) = \sum_{\bk \in [-N,N]^d} \hat{t}(\bk) \exp(2 \pi \mathrm{i} \bk \bx) = \sum_{j =1}^{J} \hat{t}(\bk_j) \exp(2 \pi \mathrm{i} \bk_j \cdot \bx),
\ee
where \(J > 4m\) is the number of non-zero frequencies of \(t\) and \((\hat{t}(\bk_j))_{j=1}^J\) is a non-increasing rearrangement of absolute values of its Fourier coefficients. For $n\in \N$ define the trigonometric polynomial $t_{n}$ as follows
\be \label{half_approx}
    t_{n}(\bx):= \sum\limits_{j=1}^{n}\hat{t}(\bk_j) \exp(2 \pi \mathrm{i} \bk_j \cdot \bx)\,.
\ee
It represents the straightforward approximation of $t$ where the $n$ largest coefficients are used. 
For our second step we follow a similar strategy as was done in \cite[Lem.\ 2]{Mak84}. Consider a family of independent discrete complex-valued random variables \(v_j\) for \(j= 2m+1,\ldots,J\) which are constructed as follows
\begin{equation} \label{def_Rad}
        v_j = \begin{cases}
            \hat{t}(\bk_j) p_j^{-1}, & \text{with probability} \,\, p_j\,, \\
            0,& \text{with probability} \,\, 1-p_j\,, \\
        \end{cases}
\end{equation}

 \noindent where 
 \(p_j = m\vert \hat{t}(\bk_j) \vert \Vert t-t_m \Vert_{\ca_1}^{-1} \).
Since \(j \geq 2m+1\) there are at least \(m\) summands in \(\Vert t-t_m \Vert_{\ca_1}\) that are not smaller than \(\vert \hat{t}(\bk_j) \vert \) and therefore it follows that the \(p_j\) are actually probabilities, i.e., that \(0 \leq p_j \leq 1\)\,.
Now we can analyze an $L_q$-approximation of \(t-t_{2m}\) by \(\sum_j v_j\exp(2\pi \mathrm{i} 
\bk_j\cdot \bx)\)
as
\begin{align} \label{EtoP}
\begin{split}
\E \Vert \bw \Vert_{L_q}^q 
&  \coloneqq \E \Big\Vert \sum_{j=2m+1}^J (\hat{t}(\bk_j) - v_j)\exp(2 \pi \mathrm{i} \bk_j \cdot \bx)  \Big\Vert_{L_q}^q \\
& \leq  \int_{\tor^d} \E \Big\vert \sum_{j=2m+1}^J (\hat{t}(\bk_j) - v_j)   \exp(2 \pi \mathrm{i} \bk_j\cdot \bx) \Big\vert^q \, \mathrm{d}\bx \,. \\
\end{split}
\end{align}
We are now ready to apply the Bernstein inequality Corollary \ref{CBernstein} to \(X_j = (\hat{t}(\bk_j) - v_j) \exp(2 \pi \mathrm{i} \bk_j\cdot \bx)\). First note that for any $\bx \in \tor^d$ we have \(\E X_j = 0\) and \(\big\vert X_j \big\vert \leq \vert\hat{t}(\bk_j)\vert p_j^{-1}  \). 
We also need a bound on the moments of the \(X_j\)
\begin{align}
\begin{split}
\E \big\vert(\hat{t}(\bk_j) - v_j)\exp(2 \pi \mathrm{i} \bk_j\cdot \bx)\big\vert^r
&=  \E \big\vert\hat{t}(\bk_j) - v_j\big\vert^r \\
& =  \vert \hat{t}(\bk_j) \vert^r (p_j^{-1} -1 )^r p_j + \vert \hat{t}(\bk_j) \vert^r(1-p_j) \\
& = \vert \hat{t}(\bk_j) \vert^r \big( (1 - p_j)^r p_j^{1-r} + 1 -p_j \big) \\
& \leq 2 \vert\hat{t}(\bk_j) \vert^r p_j^{1-r}.
\end{split}
\end{align}
In particular, for the second moment we even get \(\E \big|\hat{t}(\bk_j) - v_j\big\vert^2  \leq \vert\hat{t}(\bk_j) \vert^2 p_j^{-1} \).

We can use Lemma \ref{Bernstein} or Corollary \ref{CBernstein}, the Bernstein inequality, with 
\[B =  \sup_j \vert\hat{t}(\bk_j) p_j^{-1} \exp(2 \pi \mathrm{i} \bk_j\cdot \bx) \vert \leq m^{-1} \Vert t-t_m \Vert_{\ca_1},\] 
and 
\[
\sigma^2 = \sum_{j=2m+1}^J \vert\hat{t}(\bk_j) \vert^2 p_j^{-1} \leq  m^{-1}\Vert t-t_m \Vert_{\ca_1} \sum_{j=2m+1}^J \vert\hat{t}(\bk_j) \vert \leq  m^{-1}\Vert t-t_m \Vert_{\ca_1}^2\,. \]
We have
\begin{equation}
    \p \Big( \Big\vert \sum_{j=2m+1}^J(\hat{t}(\bk_j) - v_j) \exp(2 \pi \mathrm{i} \bk_j\cdot \bx)\Big\vert > s \Big) \leq 4\exp\Big(- \frac{s^2/4}{ m^{-1} ( \Vert t-t_m \Vert_{\ca_1}^2 + \Vert t-t_m \Vert_{\ca_1} s/(3\sqrt{2}))} \Big).
\end{equation}
Using Corollary \ref{Mo_to_Ta} with \(\beta = 4\), \(\sigma = 2 m^{-1/2} \Vert t-t_m \Vert_{\ca_1}\) and \(B = 2/3 \sqrt{2}\, m^{-1} \Vert t-t_m \Vert_{\ca_1} \) we obtain for \eqref{EtoP} 

\begin{align} \label{nq}
\begin{split}
\E \Vert \bw \Vert_{L_q} &\leq 4^{1/q} 2 m^{-1/2} \sqrt{q} \Vert t-t_m \Vert_{\ca_1}  + 2 \cdot 4^{1/q} 2/3 \sqrt{2} m^{-1} q \Vert t-t_m \Vert_{\ca_1}  \\
& \leq 4 m^{-1/2} \sqrt{q} \Vert t-t_m \Vert_{\ca_1}  + \frac{8 \sqrt{2}}{3} m^{-1} q \Vert t-t_m \Vert_{\ca_1}  \\
& \leq  \Big(4 + \frac{8\sqrt{2}}{3} \Big) \sqrt{q} m^{-1/2} \Vert t-t_m \Vert_{\ca_1} \, ,\\
\end{split}
\end{align}
where we used that \(q \geq 2\) and $m\geq q$.
The expected number of non-zero coefficients \(c_m\) is 

\begin{equation}
    \E c_m = \sum_{j=2m+1}^J p_j = \sum_{j=2m+1}^J m \vert \hat{t}(\bk_j) \vert\|t-t_m\|_{\ca_1}^{-1} \leq m.
\end{equation}
Since both random variables \(c_m\) and \(\Vert \bw \Vert_{L_q}\) are non-negative we can apply Markov's inequality to see that for any \(\delta > 0\)
\begin{equation}
    \p\big( c_m \geq (2+\delta) \E c_m \big) \leq \frac{\E c_m}{(2+\delta) \E c_m} < \frac{1}{2}.
\end{equation} 
This means  \(\p \big(c_m > 2 \E c_m) < \frac{1}{2} \) and analogously \(\p \big(\Vert \bw\Vert_{L_q} > 2 \E_v \Vert \bw\Vert_{L_q}) < \frac{1}{2} \). Therefore, the probability for either of these events happening is strictly less than one. Hence, there exists a realization  such that \(c_m \leq 2 \E c_m \) and \( \Vert \bw \Vert_{L_q} \leq 2 \E \Vert \bw \Vert_{L_q}\). 
This implies the existence of $s \in \Sigma_{2m}$ with 
\begin{equation}
    \|t-t_{2m}-s\|_{L_q} \leq 2\Big(4 + \frac{8\sqrt{2}}{3} \Big)  \sqrt{\frac{q}{m}} \Vert t-t_m \Vert_{\ca_1}\,
\end{equation} 
if $m\geq q$. So there is always $\tilde{s} \in \Sigma_{4m}$ such that
$$
    \|t-\tilde{s}\|_{L_q} \leq \Big(8 + \frac{16\sqrt{2}}{3} \Big) \sqrt{\frac{q}{m}} \Vert t-t_m \Vert_{\ca_1} \leq C \sqrt{q}\,m^{1/2-1/\theta}\|t\|_{\ca_\theta}\,,
$$
where in the last step we again applied Lemma \ref{stechkin}\,. This proves \eqref{thm:4.4} for $\ca_\theta \cap \Span \mathcal{T}^d$\,.

{\em Step 2. }Let us see that this also holds for all functions \(f \in \ca_\theta\). For each  \(f\) with $\|f\|_{\ca_\theta} \leq 1$ and \(\varepsilon >0\) there exists \(N_0 \in \N\) such that for all \(N > N_0\)
\be
\Vert f-f_N\Vert_{L_q} \leq \|f-f_N\|_{\ca_1} < \varepsilon\,.
\ee
Here $f_N$ is defined as above as the trigonometric polynomial with the largest coefficients. 
This implies the existence of $s = s(\varepsilon,N) \in \Sigma_{4m}$ such that   
\begin{align}
\begin{split}
\|f-s\|_{L_q} &\leq  \|f-f_N\|_{L_q}  + \|f_N-s\|_{L_q} \\
& \leq \varepsilon + C \, \sqrt{q} \, m^{1/2-1/\theta} \Vert f_N \Vert_{\ca_\theta} \\
& \leq \varepsilon +C \, \sqrt{q} \, m^{1/2-1/\theta} \Vert f \Vert_{\ca_\theta}\,. 
\end{split}
\end{align}
Since the second summand is depending neither on $\varepsilon$ nor on $N$ we can pass to the limit $\varepsilon \to 0$ and obtain
$$
    \sigma_{4m}(f;\mathcal{T}^d)_{L_q}\leq  C \, \sqrt{q} \, m^{1/2-1/\theta} \Vert f \Vert_{\ca_\theta}.
$$ The desired result now follows after taking \(\sup_{\Vert f \Vert_{\ca_\theta} \leq 1}\) on both sides. \bigskip

As for the lower bound we may of course consider the univariate situation $
\sigma_{m}(\ca_\theta;\mathcal{T})_{L_q(\tor)}\,.$ Let us consider the following fooling function. Fix $m\in \N$ and $j \in \N$ such that $2^j \leq m \leq 2^{j+1}$. We use a Dirichlet kernel type building block  
$$
 f = \frac{1}{(3m)^{1/\theta}} \sum\limits_{2^j \leq k < 2^{j+2}} \exp(2\pi \mathrm{i} kx)\,.
$$
Clearly, we have  
 $$
 \|f\|_{\ca_\theta} 
 = \frac{1}{(3m)^{1/\theta}} \bigg\|\sum\limits_{2^j \leq k < 2^{j+2}} \exp(2\pi \mathrm{i} k\cdot)\bigg\|_{\ca_\theta} 
 = \frac{1}{(3m)^{1/\theta}} (3\cdot 2^j)^{1/\theta} \leq 1 .
 $$

Now, given an arbitrary set \(K \subset \zz\) of cardinality \(\# K \leq m\) and arbitrary coefficients \((a_k)_{k \in K}\) we set \(g =   (3m)^{-1/2} \sum_{k \in [2^j, 2^{j+2}) \setminus K} \exp(2 \pi \mathrm{i} k x) \) and \(h =  \sum_{k \in K} a_{k} \exp(2 \pi \mathrm{i} k x)\). For these functions we get
 \begin{equation} \label{hup}
\|g\|_{L_2} 
=  (3m)^{-\frac{1}{2}} \big(\#[2^j;  2^{j+2}) - \# \big( [2^j;  2^{j+2}) \cap K \big) \big)^{\frac{1}{2}} 
\leq (3m)^{-\frac{1}{2}} \big(\#[2^j;  2^{j+2})\big)^{\frac{1}{2}} 
 %
 %
 \leq 1
 \end{equation} 
 and, therefore, 
 \begin{equation}  \label{tangle}
 \langle f-h,g \rangle  = \langle f,g \rangle 
 \geq \frac{3^{-1/\theta}}{\sqrt{3}} m^{-1/\theta-1/2} \sum_{k \in  [2^j;  2^{j+2})  \setminus K} 1 
 \geq  \frac{3^{-1/2-1/\theta}}{2} m^{-(1/\theta-1/2)} ,
 \end{equation}
 due to 
 \begin{equation} \label{2^j;2^j+2_setminus_K}
 \sum_{k \in  [2^j;  2^{j+2})  \setminus K} 1
 =
 \# [2^j;  2^{j+2}) - \# \big( [2^j;  2^{j+2}) \cap K \big)  \geq 3\cdot 2^j - m \geq 3m/2 - m = m/2 ,
 \end{equation}  
where $2^j \leq m \leq 2^{j+1}$, \(\# K \leq m\). 

 Using the Cauchy-Schwarz inequality, this implies 
  \begin{equation} 
  \frac{3^{-1/2-1/\theta}}{2} m^{-(1/\theta-1/2)} \leq \langle f-h,g \rangle \leq \|f - h\|_{L_2}\, \|g\|_{L_2} \leq \|f - h\|_{L_2}
  \leq \|f - h\|_{L_q}
 \end{equation}
and, therefore, proves  the lower bound  \eqref{thm:4.4_1}. 
\end{proof}

\begin{rem*} \label{rem_Tem}
    Results of the type of Theorem \ref{tracAq} have been shown in \cite[Theorem 2.6]{Tem15} and \cite[Theorem 9.2.8]{Tem18_book} based on greedy methods. However the dependence of the constants on the underlying dimension \(d\) was not determined. In contrast we provided a self-contained and elementary proof that gives precise constants and shows that the right-hand side of \eqref{thm:4.4} is independent of \(d\).
\end{rem*}

While Theorem \ref{tracAq} does not hold for \(q = \infty\) we can still recover a result for trigonometric polynomials in that regime.

\begin{theorem} \label{A_infty} Let \(0 < \theta \leq 1\). For a trigonometric polynomial \(t \in\mathcal{T}\big( Q\big)\) with \( Q = \bigtimes_{j =1}^d [A_j,B_j]\) and integers \( A_j  < B_j\), \(j=1,\ldots,d\) then it holds for all \(m \geq 4\log(d) \log(\# Q)\)
 it holds,
\be
\sigma_{4m}(t;\mathcal{T}^d)_{L_\infty} \leq    C_2  \, m^{1/2-1/\theta} \big(\log(d)\log ( {\# Q} ) \big)^{1/2}\Vert t\Vert_{\ca_\theta},
\ee
with an absolute constant \(C_2 < 2 e^{9/4} C < 295\), where $C$ is the constant from Theorem \ref{tracAq} and \( {\# Q} = \prod_{j=1}^d (B_j -A_j + 1) \).
\end{theorem}

\begin{proof} Let us consider the polynomial 
 $$
    \widetilde{t}(\bx) = t(\bx)\cdot \exp(-\pi \mathrm{i} (\mathbf{B}+\mathbf{A})\cdot \bx )
 $$
instead of $t(\cdot)$. Assume without loss of generality, that \({(\mathbf{B}+\mathbf{A})}/{2}\) is a vector of integers (otherwise add a $1$ in the respective component). Clearly, we have $\sigma_{n}(t;\mathcal{T}^d)_{L_\infty} = \sigma_{n}(\widetilde{t};\mathcal{T}^d)_{L_\infty}$ and the frequency domain of $\widetilde{t}(\cdot)$ is contained in $[-(B_1-A_1)/2,(B_1-A_1)/2]\times...\times [-(B_d-A_d)/2,(B_d-A_d)/2]$\,. Let us put $N_j = \lfloor (B_j-A_j+1)/2 \rfloor$, $j=1,...,d$\,. 
We will use the sharpened Nikol'skii inequality in Theorem \ref{Niko} together with the $L_q$-result in Theorem \ref{tracAq}.
We have
  \begin{align*}
    \sigma_{4m}(\widetilde{t};\mathcal{T}^d)_{L_\infty} & \leq \Vert \widetilde{t} - \widetilde s\Vert_{L_\infty} \\
    & \leq e^{1-2/q}  (2d+2)^{d/q} 2^{-d/q} {\# Q}^{1/q} \Vert t - s\Vert_{L_q}\\
    & \leq C e^{1-2/q} \, \sqrt{q} m^{1/2-1/\theta}  (d+1)^{d/q}{\# Q}^{1/q}\Vert t \Vert_{\ca_\theta}, \\
\end{align*}
 where we  applied Theorem \ref{tracAq} and noted that \(s \in \mathcal{T}\big([-\bN,\bN]\big)\cap \Sigma_{4m}\) in the proof of Theorem \ref{tracAq}. For \(d>2\) we can now choose \(q = 4 \log(d) \log({\# Q})\) and obtain
\begin{align} \label{qlarge}
\begin{split}
    \sigma_{4m}(t;\mathcal{T}^d)_{L_\infty} & \leq C e \, \sqrt{{4 \log(d) \log({\# Q})}} m^{1/2-1/\theta} \Vert t \Vert_{\ca_\theta}(d+1)^{\frac{d}{4 \log(d) \log({\# Q})}} {\# Q}^{\frac{1}{4 \log(d) \log({\# Q})}} \\
    & \leq 2 C e^{1+1/(4 \log(d))} \log(d)^{1/2} (d+1)^{1/({2\log(d)})} \log({\# Q})^{1/2} m^{1/2-1/\theta} \Vert t \Vert_{\ca_\theta} \\
    & \leq 2 C e^{2+1/(4\log(d))} \log(d)^{1/2} \log({\# Q})^{1/2} m^{1/2-1/\theta} \Vert t \Vert_{\ca_\theta}\,. 
\end{split}
\end{align} 
We used that \(d < 2\log({\# Q})\) and also 

\begin{equation} \label{est-d}
    (d+1)^{1/(2\log(d))} = e^{\frac{\log(d+1)}{2 \log(d)}}  \leq e.
\end{equation}
For \(d\leq 2\) we simply choose \(q = 4 \log({\# Q})\) and the bound still holds.
\end{proof}

  Now we turn to the regime \(1<\theta\le 2\), more precisely, to the case
 \(1< \theta<q'=\frac{q}{q-1}\) when \(2\le q<\infty\).
  To show results here we will employ an interpolation method. In \cite{TeUl22} the authors developed a framework to apply this to best \(m\)-term widths that we will make use of here. In particular we now apply {\cite[Theorem 3.3]{TeUl22}} to prove the following.
\begin{prop}\label{interpol_sigma}
    Let \(0 < \eta < 1\) and \(q > 1\). We set
    \begin{equation}
        \frac{1}{\theta} = \frac{1 - \eta}{1} + \frac{\eta}{q}.
    \end{equation}
    Then we know that \(\ca_\theta\) interpolates between \(\ca_1\) and \(\ca_q\) and it holds that,
    \begin{equation}
        \sigma_{m_1 + m_2}(\ca_\theta)_X \leq 3 \sigma_{m_1}(\ca_1)^{1-\eta}_X \sigma_{m_2}(\ca_q)^\eta_X. 
    \end{equation}
\end{prop}

\begin{proof}
    We need to ensure that \(\ca_\theta\) is of \(K_\eta\) type regarding \(\ca_1\) and \(\ca_q\) i.e. that 
    \begin{equation} \label{Ktype}
        \sup_{t>0} t^{-\eta} \inf_{f = f_1 + f_2}\big( \Vert f_1 \Vert_{\ca_1} + t\Vert f_2\Vert_{\ca_q}\big) \leq C \Vert f \Vert_{\ca_\theta},
    \end{equation}
    for some absolute constant \(C\) and any \(f \in \ca_\theta\).
    We may rewrite this formula as follows,
       \begin{equation}
        \sup_{t>0} \inf_{a = a_1 + a_2} t^{-\eta} \Vert a_1 \Vert_{\ell_1} + t^{1-\eta} \Vert a_2\Vert_{\ell_q}\leq C \Vert a \Vert_{\ell_\theta}.
    \end{equation}
    For \(t \leq 1\) this is trivial. Let now \(t>1\). Assume, without loss of generality, that the sequence \(a\) is arranged in non-increasing order. We start by estimating
    \begin{equation}
        \inf_{a = a_1 + a_2} t^{-\eta} \Vert a_1 \Vert_{\ell_1} + t^{1-\eta} \Vert a_2\Vert_{\ell_q} \leq t^{-\eta} \Vert a_{m} \Vert_{\ell_1} + t^{1-\eta} \Vert a_{-m}\Vert_{\ell_q}.
    \end{equation}
    Where we denoted by \(a_m\) the restriction of \(a\) to its first \(m\) elements and by \(a_{-m} = a - a_m\). We choose \(m = m(t) = t^{\frac{1-\eta}{\frac{1}{\theta}-\frac{1}{q}}} = t^{\frac{q}{q-1}}\). 
    We now see that by Stechkin's inequality, Lemma \ref{stechkin}, it holds that
    \begin{equation}
        \Vert a \Vert_{\ell_\theta} \geq m^{1/\theta-1/q} \sigma_m(a)_{\ell_q} = t^{{1-\eta}}\Vert a_{-m} \Vert_{\ell_q}.
    \end{equation}
    In addition, we can use the fact that \(a_m\) has no more than \(m\) non-zero entries to obtain the lower bound
    
    \begin{equation}
            \Vert a_m \Vert_{\ell_\theta}  \geq m^{\frac{1}{\theta}-1} \Vert a_m \Vert_{\ell_1}  = t^{\frac{q}{q-1} (\frac{1}{\theta}-1)} \Vert a_m \Vert_{\ell_1} = t^{\frac{q}{q-1} \eta \frac{1-q}{q}} \Vert a_m \Vert_{\ell_1} =t^{-\eta} \Vert a_m \Vert_{\ell_1}.
 \end{equation}
 This implies \eqref{Ktype} for \(C = 3\) (due to taking into account that \(m\) has to be an integer) and concludes the proof.
\end{proof}

The following embedding is a direct consequence of the well-known Hausdorff-Young inequality \cite[Satz II]{Haus23}.
\begin{prop}[Hausdorff-Young]\label{Haus_young}
    Let  \(2 \leq q \leq \infty\) and \(\frac{1}{q} + \frac{1}{q'} = 1\). Then it holds that 
    \begin{equation}
        \ca_{q'} \hookrightarrow L_q 
    \end{equation}
    and in particular the norm of the embedding is less or equal \(1\) such that
    \begin{equation}
        \sigma_0(\ca_{q'})_{L_q} \leq 1.
    \end{equation}
\end{prop}

We can now combine Proposition \ref{interpol_sigma} and Proposition \ref{Haus_young} to get the following result that extends Theorem \ref{tracAq} to the setting where \( 1 < \theta \leq 2\).


\begin{satz}\label{extend_trac}
  Let \(2 \leq q < \infty\) and \(1 < \theta < q'= \frac{q}{q-1}\). Then it holds that
  \begin{equation}\label{thm:4.4_interpol}
      \sigma_{4m}(\ca_\theta;\mathcal{T}^d)_{L_q} \leq C_1 \cdot \Big(\frac{m}{q}\Big)^{-\frac{q}{2}\left(\frac{1}{q}+\frac{1}{\theta}-1\right)}.
  \end{equation} 
  with some constant \(C_1  < 3C < 47 \), where \(C\) is the constant from Theorem \ref{tracAq} above.

 Under the same conditions we also get the following lower bound \begin{equation}\label{thm:4.4_interpol__}
      \sigma_{m}(\ca_\theta;\mathcal{T}^d)_{L_q} \geq
      \frac{1}{ 6 (\sqrt{3}\, q + 1) }
      %
      %
 m^{ -\frac{q}{2}\big(\frac{1}{q}+\frac{1}{\theta}-1\big)} .
  \end{equation}  

\end{satz}

\begin{proof}
 We set
    \begin{equation}
        \eta = \frac{1- \frac{1}{\theta}}{1-\frac{1}{q'}} < 1.
    \end{equation}
    From Proposition \ref{Haus_young} we know that \(\sigma_0(\ca_{q'})_{L_q} \leq 1\). Theorem \ref{tracAq} on the other hand gives us \(\sigma_{4m}(\ca_1)_{L_q} \leq C \sqrt{q} m^{-1/2}\). We can now, after checking that the choice of \(\eta\) and \(\theta\) agree, apply Proposition \ref{interpol_sigma} (with \(q'\) instead of \(q\)) to combine these bounds as follows
    \begin{align}
        \begin{split}
             \sigma_{4m}(\ca_\theta)_{L_{q}} & \leq 3 \sigma_{4m}(\ca_1)_{L_{q}}^{1-\eta} \sigma_0(\ca_{q'})_{L_{q}}^\eta \\
             & \leq 3\big(C \sqrt{q} m^{-1/2}\big)^{1-\eta} \\
             &= 3C^{1-\eta} \cdot \Big(\frac{q}{m}\Big)^{\frac{1-\eta}{2}} \\
             & = C_1 \cdot \Big(\frac{m}{q}\Big)^{ -\frac{q}{2}\big(\frac{1}{q}+\frac{1}{\theta}-1\big)} 
        \end{split}
    \end{align}
with \(C_1 = 3C^{1-\eta} < 3C =24 + 16 \sqrt{2} <  47\).
For \(\theta =1\) this result gives the same bound as Theorem \ref{tracAq} again (with different constants).

Before we deal with the proof of the lower bound \eqref{thm:4.4_interpol__} let us prove the inequality
 \begin{equation} 
 \label{||D_N||_leq}
\|D_N\|_{L_{q'}} 
  \leq q \, (2N+1)^{\frac{1}{q}} , \ \ \ 2 \leq q < \infty ,
 \end{equation}
where 
 \begin{equation} 
 \label{D_N=sin((2N+1)pi(x))/sin(pi(x))}
D_N(x) = \sum\limits_{k=-N}^{N} \exp(2\pi \mathrm{i} kx) = \frac{\sin\frac{2N+1}{2}2\pi x}{\sin\frac{2\pi x}{2}} 
= \frac{\sin (2N+1) \pi x}{\sin \pi x } .
 \end{equation}

 For the convenience of the reader we give a self-contained proof of \eqref{||D_N||_leq}. We follow the approach in \cite{Ul_1968}.
 Combining $\|D_N\|_{L_{\infty}}\leq 2N+1$ with \eqref{D_N=sin((2N+1)pi(x))/sin(pi(x))}, for $1 < p < \infty$, we obtain

  \begin{align}
  \label{||D_N||_p_leq}
      \begin{split}
    \|D_N\|^{p}_{L_{p}} &  \leq \|D_N\|_{L_{\infty}}^{p}\, \int\limits_{|x|\leq \frac{1}{2(2N+1)}}  \mathrm{d}x
+  \int\limits_{\frac{1}{2(2N+1)}\leq|x|\leq\frac{1}{2}} \left|\frac{1}{2x} \right|^{p}\, \mathrm{d}x \\
& < \frac{p}{p-1} (2N+1)^{p-1} .\\
      \end{split}
  \end{align}

Therefore, having \eqref{||D_N||_p_leq} (with \(q' = p\)), we complete the proof of \eqref{||D_N||_leq}:
 $$
\|D_N\|_{L_{q'}} 
 \leq  \left(\frac{q'}{q'-1} (2N+1)^{q'-1}\right)^{\frac{1}{q'}} 
 = q^{1-\frac{1}{q}} (2N+1)^{\frac{1}{q}}
  \leq q \, (2N+1)^{\frac{1}{q}} .
 $$

As for the lower bound we may of course consider the univariate situation $
\sigma_{m}(\ca_\theta;\mathcal{T})_{L_q(\tor)}\,.$ Let us consider the following fooling function. Fix $m\in \N$ and $j \in \N$ such that  \begin{equation} \label{2^j_leq_m^q/2}
2^j \leq m^{q/2} \leq 2^{j+1}.
 \end{equation}
 We use a Dirichlet kernel type building block  
$$
    f = \frac{1}{(3m^{q/2})^{1/\theta}}\sum\limits_{2^j \leq k < 2^{j+2}} \exp(2\pi \mathrm{i} kx)\,.
 $$
Then we have  
 $$
 \|f\|_{\ca_\theta} 
 = \frac{1}{(3m^{q/2})^{1/\theta}} \bigg\|\sum\limits_{2^j \leq k < 2^{j+2}} \exp(2\pi \mathrm{i} k\cdot)\bigg\|_{\ca_\theta} 
 = \frac{1}{(3m^{q/2})^{1/\theta}} (3\cdot 2^j)^{1/\theta} \leq 1 
 $$
due to the above assumption  \eqref{2^j_leq_m^q/2}. 
It also holds 



 \begin{equation} \label{||D_n||_q'__leq}
 \bigg\|\sum\limits_{2^j \leq k < 2^{j+2}} \exp(2\pi \mathrm{i} k\cdot)\bigg\|_{L_{q'}}
 \leq q (3 \cdot 2^j)^{1-1/q'}
 \leq 3^{1/q} q \, m^{1/2} 
 \leq 3^{1/2} q \, m^{1/2} ,
 \end{equation}
according to \eqref{||D_N||_leq} and \eqref{2^j_leq_m^q/2}.


Now, given an arbitrary set \(K \subset \zz\) of cardinality \(\# K \leq m\) and arbitrary coefficients \((a_k)_{k \in K}\) we set \(g =  (3^{1/2} q + 1)^{-1} m^{-1/2} \sum_{k \in [2^j, 2^{j+2}) \setminus K} \exp(2 \pi \mathrm{i} k x) \) and \(h =  \sum_{k \in K} a_{k} \exp(2 \pi \mathrm{i} k x)\). For these functions, by using triangular inequality, $\|\cdot\|_{L_{q'}}\leq \|\cdot\|_{L_{2}}$, $1< q'\leq 2$, \eqref{||D_n||_q'__leq} and \(\# K \leq m\), we get
 \begin{align}\label{hup_q/2}
 \begin{split}
\hspace{-15pt}\|g\|_{L_{q'}}  & =  (3^{1/2} q + 1)^{-1} m^{-1/2} \Big\|\sum\limits_{2^j \leq k < 2^{j+2}} \exp(2\pi \mathrm{i} k\cdot) \hspace{5pt} - \hspace{-7pt}\sum_{k \in [2^j, 2^{j+2}) \cap K} \exp(2 \pi \mathrm{i} k \cdot)\Big\|_{L_{q'}} \\
&  \leq (3^{1/2} q + 1)^{-1}  m^{-1/2} \bigg(\Big\| \hspace{-1pt}\sum\limits_{2^j \leq k < 2^{j+2}} \hspace{-10pt}\exp(2\pi \mathrm{i} k\cdot)\Big\|_{L_{q'}} \hspace{-5pt} + \Big\| \hspace{-2pt}\sum_{k \in [2^j, 2^{j+2}) \cap K} \hspace{-10pt} \exp(2 \pi \mathrm{i} k \cdot)\Big\|_{L_{2}} 
\bigg) \\
 & \leq (3^{1/2} q + 1)^{-1}  m^{-1/2} \Big(
 3^{1/2} q \, m^{1/2} + \big(\# \big([2^j;  2^{j+2}) \cap K\big)\big)^{1/2}
\Big) \\
& \leq 1 . \\
 \end{split}
\end{align} 
 Using  \eqref{hup_q/2}, H\"older's inequality, and
 \begin{equation}\label{sum_1_geq_mq/2} 
 \sum_{k \in  [2^j;  2^{j+2})  \setminus K} 1
 \geq 
 \# [2^j;  2^{j+2}) - \# K  \geq 3\cdot 2^j - m \geq 3m^{q/2}/2 - m^{q/2} = m^{q/2}/2 
 \end{equation}  
we obtain
\begin{align}\label{Lower_bound_Wiener_1<theta<q'}  
  \begin{split}
 \|f - h\|_{L_q}  & \geq \|f - h\|_{L_q}\, \|g\|_{L_{q'}} \geq \langle f-h,g \rangle = \langle f,g \rangle \\
& = (3^{1/2} q + 1)^{-1} m^{-1/2} (3m^{q/2})^{-1/\theta}  \sum_{k \in  [2^j;  2^{j+2})  \setminus K} 1 \\
  & \geq \frac{1}{ 6 (\sqrt{3}\, q + 1) }  m^{{ -\frac{q}{2}\big(\frac{1}{q}+\frac{1}{ \theta}}-1\big)} 
   \end{split}
 \end{align}
this proves the lower bound in \eqref{thm:4.4_interpol__}.
\end{proof}

The space \(\ca_\theta\) with \(\theta > 1\) cannot be embedded into \(L_\infty\). Therefore it is not possible to extend the results in this regime to measuring the error in \(L_\infty\).

\section{Application -- Besov spaces with mixed smoothness}

The results in the previous section may be applied to Besov spaces with dominating mixed smoothness as considered for instance in \cite[Page 40, (3.3.3)]{DuTeUl18}. Consequences for weighted Wiener classes of several types will be given in the forthcoming paper \cite{MSU25}. For Besov spaces with mixed smoothness we give both, bounds for best $m$-term widths and (non-linear) sampling numbers via \eqref{eq:intro1}. Let us fix an equivalent quasi-norm and start from the segmentation \(\zz = \bigcup_{n=0}^\infty I_n\) 
with $I_0 = \{0\}$ and for $n\in \mathbb{N}$
 $$
 I_n = \big\{k \in \zz ~:~  2^{n-1} \leq \vert k\vert < 2^{n}\big\}\,.
 $$
This gives \(\#I_n = 2^{n}\).
For \( \bk \in \nd \) we write
 \be\label{Ik}
    I_{\bk} = I_{k_1} \times \ldots \times I_{k_d},
 \ee
each of these blocks contains 
 \be \label{Vol_Ik}
 \#I_{\bk} = 2^{\vert\bk \vert_1}  
 \ee
points. This also immediately induces a segmentation of \(\Z\) into cuboids \eqref{Ik}, that is \(\Z = \bigcup_{\bk \in \nd} I_{\bk} \).

For \(1 < p < \infty\), \(0 < \theta < \infty\) and \(r > 0\) we define the (periodic) Besov space \(\mathbf{B}^r_{p,\theta}\) with mixed smoothness,
\be \label{def_Besov}
\mathbf{B}^r_{p,\theta} \coloneqq \Big\{ f \in L_p(\tor^d):~  \Vert f \Vert_{\mathbf{B}^{r}_{p,\theta}} < \infty \Big\} \, ,
\ee
with the (quasi-)norm
\be \label{def_Besov_norm}
\Vert f \Vert_{ \mathbf{B}^r_{p,\theta}} \coloneqq \bigg( \sum_{\bj \in \nd }  2^{\vert \bj\vert_1 r \theta} \Big\Vert \sum_{\bk\in I_\bj } \hat{f}(\bk)\exp(2\pi \mathrm{i} \bk \bx)\Big\Vert_{L_p}^\theta\bigg)^{\frac{1}{\theta}}.
\ee 
and the typical modification if $\theta = \infty$\,.

In the regime \(q<\infty\) the above results from section \ref{sec:4} imply the following sharp bounds. 
 \begin{satz}\label{thm:5.3} 
Let  $2\leq p<\infty$, $0<\theta \leq 1$ and $2\leq q \leq m <\infty$.\\ 
{\em (i)} It holds
\be\label{thm:4.4:Besov}
\sigma_{4m}(\mathbf{B}^{1/\theta-1/2}_{p,\theta};\mathcal{T}^d)_{L_q} \leq C \sqrt{q}\, m^{-(1/\theta-1/2)}\,,
\ee
where  $C =\frac{24 + 16 \sqrt{2}}{3} <  16$.\\
{\em (ii)} In addition, we get for $2\leq q \leq \infty$ 
\be\label{eq:5.8:Besov_below}
\sigma_{m}(\mathbf{B}^{1/\theta-1/2}_{2,\theta};\mathcal{T}^d)_{L_q} \geq 6^{-1/2-1/\theta} m^{-(1/\theta-1/2)} \,.
\ee
For small \(m < q\) we again have \(\sigma_{4m}(\mathbf{B}^{1/\theta-1/2}_{p,\theta};\mathcal{T}^d)_{L_q} \leq C q\, m^{-(1/\theta)}\) instead.
 \end{satz}
 \begin{proof} The embedding
$\mathbf{B}^{1/\theta-1/2}_{p,\theta} \hookrightarrow \ca_\theta$ (mentioned in the proof of Theorem \ref{besov}, see also \cite{MSU25}) if $2\leq p<\infty$, $0<\theta \leq 1$ together with Theorem \ref{tracAq} implies the upper bound.

As for the lower bound we may of course consider the univariate situation instead of the multivariate one, that is $
\sigma_{m}(\mathbf{B}^{1/\theta-1/2}_{2,\theta};\mathcal{T})_{L_2(\tor)}\,.$ Fix $m\in \N$ and $j \in \N$ such that $2^j \leq m \leq 2^{j+1}$. Using Dirichlet-kernel-type building blocks \begin{equation}\label{Dirichlet_kernel_type_building_blocks}
D_{[2^{j+i},2^{j+i+1})} (x):= \sum\limits_{2^{j+i} \leq k < 2^{j+i+1}} \exp(2\pi \mathrm{i} kx)\,, \ \ i=0,1,   
 \end{equation} 
we consider the following fooling function
 $$
    f = (6m)^{-1/\theta} \sqrt{2} 
    \left(D_{[2^{j},2^{j+1})} (x) 
    + 
    D_{[2^{j+1},2^{j+2})} (x) \right) \, . 
 $$ 
 According to \eqref{def_Besov_norm}, we have
\begin{equation} 
    \begin{split} 
 \|f\|_{\mathbf{B}^{1/\theta-1/2}_{2,\theta}} & = (6m)^{-1/\theta} \sqrt{2} \left(\sum\limits_{i=0}^{1}2^{(j+i+1)(1/\theta-1/2)\theta} \left\|D_{[2^{j+i},2^{j+i+1})}\right\|_{L_{2}}^{\theta} \right)^{1/\theta}
 \\
 & = (3m)^{-1/\theta} ( 2^{j} + 2^{j+1} )^{1/\theta}
 \\
 & \leq 1 \, . 
    \end{split}
\end{equation}



 Now, given an arbitrary set \(K \subset \zz\) of cardinality \(\# K \leq m\) and arbitrary coefficients \((a_k)_{k \in K}\) we set \(g =  (3m)^{-1/2} \sum_{k\in [2^j;  2^{j+2}) \setminus K} \exp(2 \pi \mathrm{i} k x) \) and \(h =  \sum_{k \in K} a_{k} \exp(2 \pi \mathrm{i} k x)\). For these functions we have $\|g\|_{L_2}\leq 1$  according to \eqref{hup} and we see that because of \eqref{2^j;2^j+2_setminus_K}
 that
 \begin{equation}  \label{tangle2}
 \langle f-h,g \rangle  = \langle f,g \rangle 
 \geq \frac{6^{-1/\theta}\sqrt{2}}{\sqrt{3}} m^{-1/\theta-1/2} \sum_{k \in  [2^j;  2^{j+2})  \setminus K} 1 
 \geq  6^{-1/2-1/\theta} m^{-(1/\theta-1/2)} .
 \end{equation}

 Using the Cauchy-Schwarz inequality, this implies 
  \begin{equation} 
 6^{-1/2-1/\theta} m^{-(1/\theta-1/2)} \leq \langle f-h,g \rangle \leq \|f - h\|_{L_2}\, \|g\|_{L_2} \leq \|f - h\|_{L_2} \leq \|f - h\|_{L_q}
 \end{equation}
and, therefore, proves    \eqref{eq:5.8:Besov_below}. 
\end{proof}

Now, based on Theorem \ref{extend_trac},  we turn to the regime 
\(1<\theta < 2\), 
more precisely, to the case 
\(1< \theta<q'=\frac{q}{q-1}\) 
for classes $\mathbf{B}^{1/\theta-1/2}_{2,\theta}$ 
when \(2\le q<\infty\).

\begin{satz}\label{extend_trac_B}
  Let \(2 \leq q < \infty\) and \(1 < \theta < q'= \frac{q}{q-1}\). Then it holds that
\begin{equation}\label{thm:4.4_interpol_B}
      \sigma_{4m}(\mathbf{B}^{1/\theta-1/2}_{2,\theta})_{L_q} \leq C_1 \cdot \Big(\frac{m}{q}\Big)^{-\frac{q}{2}\left(\frac{1}{q}+\frac{1}{\theta}-1\right)}.
  \end{equation} 
  with some constant \(C_1  < 3C < 47 \), where \(C\) is the constant from Theorem \ref{tracAq}.

 Under the same conditions we also get the following lower bound \begin{equation}\label{thm:4.4_interpol__B}
      \sigma_{m}(\mathbf{B}^{1/\theta-1/2}_{2,\theta})_{L_q} \geq
      \frac{\sqrt{2}}{ 12 (\sqrt{3}\, q + 1) }
 m^{ -\frac{q}{2}\big(\frac{1}{q}+\frac{1}{\theta}-1\big)} .
  \end{equation}  
 \end{satz}

 \begin{proof} 
The embedding 
$\mathbf{B}^{1/\theta-1/2}_{2,\theta} \hookrightarrow \ca_\theta$ (see  \cite[Thm.\ 5.3]{MSU25}) together with Theorem \ref{extend_trac} implies the upper bound.

As for the lower bound we may of course consider the univariate situation instead of the multivariate one, that is 
$\sigma_{m} (\mathbf{B}^{1/\theta-1/2}_{2,\theta};\mathcal{T})_{L_q(\tor)}\,.$ Fix $m\in \N$ and $j \in \N$ satisfying the double inequality \eqref{2^j_leq_m^q/2}. Using Dirichlet-kernel-type building blocks defined by \eqref{Dirichlet_kernel_type_building_blocks},  
we consider the following fooling function
 $$
    f = \big(6m^{q/2}\big)^{-1/\theta} \sqrt{2} 
    \left(D_{[2^{j},2^{j+1})} (x) 
    + 
    D_{[2^{j+1},2^{j+2})} (x) \right) \, . 
 $$ 
 According to \eqref{def_Besov_norm}, \eqref{2^j_leq_m^q/2}, we have
\begin{equation} 
    \begin{split} 
 \|f\|_{\mathbf{B}^{1/\theta-1/2}_{2,\theta}} & = \big(6m^{q/2}\big)^{-1/\theta}  \sqrt{2} \left(\sum\limits_{i=0}^{1}2^{(j+i+1)(1/\theta-1/2)\theta} \left\|D_{[2^{j+i},2^{j+i+1})}\right\|_{L_{2}}^{\theta} \right)^{1/\theta}
 \\
 & = \big(3m^{q/2}\big)^{-1/\theta} ( 2^{j} + 2^{j+1} )^{1/\theta}
 \\
 & \leq 1 \, . 
    \end{split}
\end{equation}

Now, given an arbitrary set \(K \subset \zz\) of cardinality \(\# K \leq m\) and arbitrary coefficients \((a_k)_{k \in K}\) we set \(g =  (3^{1/2} q + 1)^{-1} m^{-1/2} \sum_{k \in [2^j, 2^{j+2}) \setminus K} \exp(2 \pi \mathrm{i} k x) \) and \(h =  \sum_{k \in K} a_{k} \exp(2 \pi \mathrm{i} k x)\), where $\|g\|_{L_{q'}}\leq 1$ according to \eqref{hup_q/2}.  

 Following \eqref{Lower_bound_Wiener_1<theta<q'} by involving \eqref{sum_1_geq_mq/2}, this implies 
\begin{align}\label{Lower_bound_Besov_1<theta<q'}  
  \begin{split}
 \|f - h\|_{L_q}  & \geq \langle f,g \rangle \\
& = \sqrt{2} (\sqrt{3} q + 1)^{-1} m^{-1/2} (6m^{q/2})^{-1/\theta}  \sum_{k \in  [2^j;  2^{j+2})  \setminus K} 1 \\
  & \geq \frac{\sqrt{2}}{ 12 (\sqrt{3}\, q + 1)}  m^{ -\frac{q}{2}\big(\frac{1}{q}+\frac{1} {\theta}-1\big)} 
   \end{split}
 \end{align}
and, therefore, proves the lower bound.  
 \end{proof}



The results of Theorem \ref{extend_trac_B} complement the corresponding results in \cite{Rom_I_92, Romanyuk_Izv_2003, StasUMZh_2016}, since Theorem \ref{extend_trac_B} provides explicit constants for both the upper and the lower bounds.

When we pass the target space $L_q$ (in Theorem \ref{thm:5.3}) to $q = \infty$ we again get a slightly worse upper bound.

\begin{satz} \label{besov}
 Let \(m,d \in \N\) and $(p,\theta)\in \{(p,\theta) : \, 2\leq p<\infty$, $0<\theta\leq 1\} \setminus (2,1)$. Then it holds
 \begin{equation}
  \begin{split}
     &\sigma_{4m}(\mathbf{B}^{1/\theta-1/2}_{p,\theta};\mathcal{T}^d)_{L_\infty}\\ 
     &~~~\leq C_3 \,(d\log(d))^{1/2}\, \Big(\frac{1}{\theta}-\frac{1}{2}-\frac{1}{p}\Big)^{-1/2}  \max\Big\{\Big(\frac{2d}{p}\Big)^{1/2}, \Big(\frac{1}{\theta}-\frac{1}{2}\Big)^{1/2}\Big\} \, m^{-(1/\theta-1/2)} \log(dm)^{1/2}\,,
   \end{split}
\end{equation}
where \(C_3 = C_2 + 3e < 304\), and \(C_2\) is the constant from Theorem \ref{A_infty}.
\end{satz}


\begin{proof} 
For \(f \in \mathbf{B}^{1/\theta-1/2}_{p,\theta}\) we start by projecting on the space of trigonometric polynomials 
$\mathcal{T}([-N,N]^d)$ with the projection operator 
\be\label{eq105}
    P_{N}f = \sum\limits_{\bk \in [-N,N]^d} \hat{f}(\bk)\exp(2\pi \mathrm{i} \bk \bx)\,.
\ee
Obviously it holds \(P_{N}f  \in \mathcal{T}([-N,N]^d)\). The projection now allows for the following 
\begin{equation}\label{f101} 
    \begin{split}
         \sigma_{4m}  \big(\mathbf{B}^{1/\theta-1/2}_{p,\theta};\mathcal{T}^d\big)_{L_\infty} 
         & = \sup_{\|f\|_{\mathbf{B}^{1/\theta-1/2}_{p,\theta}}\leq 1} \sigma_{4m}\big(f;\mathcal{T}^d\big)_{L_\infty} \\
         & = \sup_{\|f\|_{\mathbf{B}^{1/\theta-1/2}_{p,\theta}}\leq 1} \inf_{s \in \Sigma_{4m}} \Vert f - s\Vert_{L_\infty} \\
         &  \leq \sup_{\|f\|_{\mathbf{B}^{1/\theta-1/2}_{p,\theta}}\leq 1} \Big( \inf_{s \in \Sigma_{4m} } \Vert P_Nf - s\Vert_{L_\infty} + \Vert f -P_Nf\Vert_{L_\infty} \Big).\\
    \end{split}
\end{equation}
We take care of the first summand. By H\"older's inequality it is easy to see that the space $\mathbf{B}^{1/\theta-1/2}_{p,\theta}$ is embedded into $\ca_\theta$ with norm $1$ (independently of the dimension) if $2\leq p<\infty$ and $0<\theta \leq 1$, see, e.g. \cite{MSU25}. This gives
\begin{equation} \label{f102}
    \begin{split}
        \sup_{\|f\|_{\mathbf{B}^{1/\theta-1/2}_{p,\theta}}\leq 1} \inf_{s \in \Sigma_{4m} } \Vert P_Nf - s\Vert_{L_\infty}
        &\leq \sup_{\|f\|_{\ca_\theta}\leq 1} \inf_{s \in \Sigma_{4m} } \Vert P_Nf - s\Vert_{L_\infty}\\
        &\leq C_2\,d^{1/2}\, m^{1/2-1/\theta} \log(2N)^{1/2} \log(d)^{1/2}\,,
    \end{split}
\end{equation}
where the last inequality follows from Theorem \ref{A_infty}. 

 It remains to estimate $\|f-P_Nf\|_{L_\infty}$\,. Here the condition $(p,\theta)\in \{(p,\theta) : \, 2\leq p<\infty$, $0<\theta\leq 1\} \setminus (2,1)$ will play an important role. Let us assume that $N = 2^\ell$, where $\ell\in\N$ will be specified later. It clearly holds
\begin{equation}
    \begin{split}
        \|f-P_Nf\|_{L_\infty}^{\theta} \leq \sum\limits_{|\bj|_\infty \geq \ell} \Big\|\sum_{\bk \in I_{\bj}}\hat{f}(\bk)\exp(2\pi \mathrm{i} \bk \bx)\Big\|_{L_\infty}^{\theta} .
    \end{split}
\end{equation}
Here we use our refined version of Nikol'skii's inequality from Theorem \ref{Niko}. We obtain
$$
    \Big\|\sum_{\bk \in I_{\bj}}\hat{f}(\bk)\exp(2\pi \mathrm{i} \bk \bx)\Big\|_{L_\infty}^{\theta} \leq 
    e^{\theta - 2\theta/p} \big((2d+2)^d2^{|\bj|_1}\big)^{\theta/p}\Big\|\sum_{\bk \in I_{\bj}}\hat{f}(\bk)\exp(2\pi \mathrm{i} \bk \bx)\Big\|_{L_p}^{\theta}\,, 
$$
which implies 
\begin{equation}\label{eq100}
    \begin{split}
        \|f-P_Nf\|_{L_\infty}
        &
        \leq e^{1 - 2/p} (2d+2)^{d/p} \bigg( \sum\limits_{|\bj|_1 \geq \ell} 2^{|\bj|_1 \theta/p}
        \Big\|\sum_{\bk \in I_{\bj}}\hat{f}(\bk)\exp(2\pi \mathrm{i} \bk \bx)\Big\|_{L_p}^{\theta} \bigg)^{\frac{1}{\theta}}\\
        &\leq e^{1 - 2/p} (2d+2)^{d/p} \Big(\sup\limits_{|\bj|_1 \geq \ell}2^{|\bj|_1(1/p-(1/\theta-1/2))}\Big) \\
        &\times
        \bigg(\sum\limits_{\vert \bj\vert_1 \geq \ell}2^{\vert \bj\vert_1(1/\theta-1/2)\theta}
        \Big\|\sum_{\bk \in I_{\bj}}\hat{f}(\bk)\exp(2\pi \mathrm{i} \bk \bx)\Big\|_{L_p}^{\theta} \bigg)^{\frac{1}{\theta}} \\
        &\leq e^{1 - 2/p} (2d+2)^{d/p}N^{-(1/\theta-1/2-1/p)}\|f\|_{\mathbf{B}^{1/\theta-1/2}_{p,\theta}}\,.
    \end{split}
\end{equation} 
Choosing $l=\log_2(N)$ such that $(2d+2)^{d/p} N^{-(1/\theta-1/2-1/p)} \leq m^{1/2-1/\theta}$ we obtain 
 \be\label{eq101}
      \big((2d+2)^{d/p} m^{1/\theta-1/2}\big)^{(1/\theta-1/2 -1/p)^{-1}} \leq N \leq 2\,\big((2d+2)^{d/p} m^{1/\theta-1/2}\big)^{(1/\theta-1/2 -1/p)^{-1}}\,.
\ee
Together with \eqref{f101} and \eqref{f102} we see that 
\begin{equation}\nonumber
  \begin{split}
     &\sigma_{4m}(\mathbf{B}^{1/\theta-1/2}_{p,\theta};\mathcal{T}^d)_{L_\infty}\\ 
     &~~~\leq C_3 \,(d\log(d))^{1/2}\, \Big(\frac{1}{\theta}-\frac{1}{2}-\frac{1}{p}\Big)^{-1/2}  \max\Big\{\Big(\frac{2d}{p}\Big)^{1/2}, \Big(\frac{1}{\theta}-\frac{1}{2}\Big)^{1/2}\Big\} \, m^{1/2-1/\theta} \log(dm)^{1/2}\,.
   \end{split}
\end{equation}
\end{proof}


\begin{rem*}  
We are not aware of a result in the spirit of Theorem \ref{besov} in the literature, i.e.,  with limiting mixed smoothness $r=1/2$ and  $0<\theta\leq 1$. In case $d=1$ and smoothness larger than \(1/2\), the $\sqrt{\log(dm)}$-factor in Theorem \ref{besov} is not needed, see \cite[Sect.\ 6]{DeTe95}. 
\end{rem*}

\begin{rem*}[Lower bounds] {\em (i)} It should be noted, that \eqref{thm:4.4:Besov} shows the sharp asymptotic bound of $$\sigma_{m}(\mathbf{B}^{1/2}_{p,1};\mathcal{T}^d)_{L_q} \asymp m^{-1/2}$$ 
for $2\leq p<q<\infty$, which is already known and was proved in \cite[Thm.\ 1]{Rom_I_92}, \cite[Thm.\ 2.1]{Romanyuk_Izv_2003}, \cite[Thm.\ 2]{Stas_TrIMM_2017} (see also \cite[Thm.\ 3]{Belinsk_Msb_1987}, 
\cite[Thm.\ 2]{Stas_JAT_2014}
for univariate case).

Theorem \ref{thm:5.3},(i) complements the result of \cite[Thm.\ 3.2]{Romanyuk_Izv_2003} in the sense that, in case $2\leq q \leq p<\infty$, $\theta=1$, $r=1/2$ we have an explicitly specified constant in \eqref{thm:4.4:Besov}.



{\em (ii)} Note also, that the argument for the proof of Theorem \ref{thm:5.3},(ii) can be modified to treat the case $p>2$ via involving random signs. With a straightforward use of the Khintchine inequality we obtain the existence of a sign pattern $\varepsilon_k \in \{-1,1\}$ with 
$$
  f_{\varepsilon} = \frac{1}{m^{1/\theta}}\sum\limits_{2^j \leq k < 2^{j+2}} \varepsilon_k \exp(2\pi \mathrm{i} kx)\,,
$$
such that $\|f_\varepsilon\|_{\mathbf{B}^{1/\theta-1/2}_{p,\theta}} \leq B_p\,6^{1/\theta}/\sqrt{2}$, where $B_p$ denotes the upper constant in the Khintchine inequality, which is of order $\sqrt{p}$\,. For $2<p<\infty$ this results in 
$$
\sigma_{m}(\mathbf{B}^{1/\theta-1/2}_{p,\theta};\mathcal{T}^d)_{L_2}  \geq \frac{6^{-1/2-1/\theta}}{B_p} m^{-(1/\theta-1/2)}\,.
$$


\end{rem*} 
\subsection{Implications for sampling recovery}

The $L_\infty$-bounds above have immediate consequences for the sampling recovery problem in Besov spaces with mixed smoothness via \eqref{eq:intro1}. Recall that we recover from function values only (standard information), the setting going back to \cite{WaWo01}.

\begin{cor}\label{cor:besov} Let \(m,d \in \N\) and $(p,\theta)\in \{(p,\theta) : \, 2\leq p<\infty$, $0<\theta\leq 1\} \setminus (2,1)$. Then there are absolute constants $C>0$ and a constant $c=c(p,\theta)>0$ such that 
$$
    \varrho_{cd^2(\log^2d)m(\log^3m)}(\mathbf{B}^{1/\theta-1/2}_{p,\theta})_{L_2} 
    \leq C\,d\,\sqrt{\frac{(2-\theta)p}{(2-\theta)p-2\theta}}\,m^{-(1/\theta-1/2)}\log(dm)^{1/2}\,.
$$ 
\end{cor} 
\begin{proof} We employ \eqref{eq:intro1} and use Theorem \ref{besov} to estimate $\sigma_{4m}(\mathbf{B}^{1/\theta-1/2}_{p,\theta};\mathcal{T}^d)_{L_\infty}\,.$ In order to estimate 
$$
    E_{\mathcal{T}([-N,N]^d)}(\mathbf{B}^{1/\theta-1/2}_{p,\theta})_{L_\infty} \leq \sup\limits_{\|f\|_{\mathbf{B}^{1/\theta-1/2}_{p,\theta}} \leq 1} \|f-P_Nf\|_{L_\infty} \leq e m^{-(1/\theta-1/2)}
$$
we use \eqref{eq105}, \eqref{eq100} and choose $N$ analogously as in \eqref{eq101}. The index for $\varrho_n$, in particular the additional $\log m$-factor, comes from \eqref{eq:intro1} and the choice of $N$.   
\end{proof}
We can now reformulate this in a simpler but less precise notation (in particular the logarithmic rates can easily be refined).

 \begin{rem*}{\em (i)} The result in  Corollary \ref{cor:besov} can be rewritten as follows
  
$$
    \varrho_{m}(\mathbf{B}^{1/\theta-1/2}_{p,\theta})_{L_2} \leq C(p,\theta)d^{2/\theta}m^{-(1/\theta-1/2)}\log(dm)^{3/\theta - 1}\,.
$$

What concerns the order in $m$ the result might be improved by a $\sqrt{\log m}$ using the refined version of \eqref{eq:intro1} in Krieg \cite{Kr23} or the WOMP (weak orthogonal matching pursuit) proposed in Dai, Temlyakov \cite{DaTe24}. 

{\em (ii)} We obtain another particular situation, where 
$$
    \varrho_{n}(\mathbf{B}^{1/\theta-1/2}_{p,\theta})_{L_2} = o\left(d_n(\mathbf{B}^{1/\theta-1/2}_{p,\theta})_{L_2}\right)
$$
if $d$ is large, see also \cite[Cor.\ 4.16 and Rem.\ 4.17]{JUV23}. In other words, the non-linear sampling numbers $\varrho_n$ decay faster than the corresponding Kolmogorov widths $d_n$ for sufficiently large $d$. The improvement happens in the logarithm, see \cite[Thm.\ 4.5.3]{DuTeUl18}. 

{\em (iii)} Corresponding result of this remark (see (i))  in the missing case ($p=2$, $\theta=1$), more exactly, the  order bounds for $\varrho_{n}(\mathbf{B}^{1/2}_{2,1})_{L_2}$,  we can find in \cite[Thm.\ 6.4, 6.6(i)]{GlSeTi23}, that is 
 $$
   n^{-1/2} \lesssim  \varrho_{n}(\mathbf{B}^{1/2}_{2,1})_{L_2} \lesssim n^{-1/2}
   \log (n)^{3(d-1)/2} .
 $$
 
 \end{rem*}

\vspace{0.5cm}
{\bf Acknowledgements.}
  M.M. is supported by the ESF, being co-financed by the European Union and from tax revenues on the basis of the budget adopted by the Saxonian State Parliament. S.S. is supported 
  by the Philipp Schwartz Initiative of the Alexander von Humboldt Foundation.  
 S.S. acknowledges support by the German Research Foundation (DFG 403/4-1). 
  We gratefully acknowledge the support of the Leibniz Center for Informatics, where several discussions about this research were held during the Dagstuhl Seminar
``Algorithms and Complexity for Continuous Problems'' (Seminar ID 23351) in August/September 2023. T.U. would like to particularly thank D. Krieg, T. K\"uhn and E. Novak. All authors would like to thank G. Petrova, R.A. DeVore, T. K\"uhn and P. Wojtaszczyk for discussions on this subject during the Bedlewo conference ``Banach spaces for analysts'' in honor of P. Wojtaszczyk's 70th birthday. Last but not least, we would like to thank W. Sickel for pointing out to us the references  \cite{NNS22}, \cite{NN22} and \cite{DPW24} and S. Tikhonov for pointing out the reference \cite{DT05}\,.

\bibliography{literature}
\bibliographystyle{abbrv}

\end{document}